\documentclass[11pt,a4paper,reqno]{amsart}
\usepackage[T1]{fontenc}
\usepackage[utf8]{inputenc}
\usepackage{lmodern}
\usepackage{a4wide}
\usepackage{amssymb,mathtools,mathrsfs}
\usepackage{microtype}
\usepackage{enumitem}
\usepackage{xcolor}
\definecolor{linkblue}{RGB}{30,65,90}
\usepackage{hyperref}
\hypersetup{colorlinks=true,linkcolor=linkblue,citecolor=linkblue,urlcolor=linkblue,
 pdftitle={Solution of the Mumford-Shah conjecture},
 pdfauthor={Francesco Deangelis}}
\usepackage{bookmark}
\newtheorem{theorem}{Theorem}[section]
\newtheorem{lemma}[theorem]{Lemma}
\newtheorem{proposition}[theorem]{Proposition}

\theoremstyle{definition}

\theoremstyle{remark}

\newcommand{\R}{\mathbb R}

\newcommand{\dd}{\mathop{}\!\mathrm d}
\newcommand{\Rot}{\mathsf J}

\DeclareMathOperator{\supp}{supp}

\numberwithin{equation}{section}
\allowdisplaybreaks[1]
\title{Solution of the Mumford--Shah conjecture}
\author{Francesco Deangelis}
\address{Applied Mathematics M\"unster, University of M\"unster,
Einsteinstrasse 62, 48149 M\"unster, Germany}
\email{\href{mailto:francesco.deangelis@uni-muenster.de}{francesco.deangelis@uni-muenster.de}}
\date{}
\begin{document}
\begin{abstract}
We prove the Mumford--Shah conjecture, formulated in 1989, by completing
the classification of blow-up limits of minimizers in the plane.
The remaining case concerns limits whose discontinuity set is noncompact
and has connected complement. In this setting, we rule out bounded
connected components using second-variation inequalities obtained from
combined inner and outer variations associated with two perpendicular
translations of a compact subset separated from the rest of the
discontinuity set. Established results then imply that the discontinuity
set is connected, and the known classification for connected discontinuity
sets completes the proof.
\end{abstract}

\maketitle

\section{Introduction and main result}\label{sec:introduction}

The Mumford--Shah functional was introduced as a variational model for
image segmentation \cite{MS89}. Its unknowns are a function $u$, representing
a piecewise smooth approximation of a given image $g$, and a set $K$,
representing its edges. The model balances the Dirichlet energy of $u$
away from $K$, the length of $K$, and the discrepancy between $u$ and the image.
This competition leads to a free discontinuity problem: both the function
and the set across which it may jump are determined by minimization.

Let $\Omega\subset\R^2$ be open, let $g\in L^\infty(\Omega;\R)$,
and let $\lambda\geq0$. For a relatively closed, countably
$\mathcal H^1$-rectifiable set $K\subset\Omega$ and a function
$u\in W^{1,2}_{\mathrm{loc}}(\Omega\setminus K;\R)$, we consider
\begin{equation}\label{intro:energy}
 \mathcal E_\lambda(u,K;U)
 =\int_{U\setminus K}\bigl(|\nabla u|^2+\lambda|u-g|^2\bigr)\,\dd x
        +\mathcal H^1(K\cap U)
\end{equation}
for bounded open sets $U\subset\Omega$. The term
$\lambda\int_{U\setminus K}|u-g|^2\,\dd x$ is called the
\emph{fidelity term}: it penalizes the discrepancy between $u$ and
the prescribed datum $g$, with weight $\lambda$.
The Dirichlet term penalizes large spatial gradients, whereas
the length term makes discontinuities costly; allowing jumps can,
however, reduce the cost of approximating sharp transitions in $g$.

The central regularity question concerns the geometry of the minimizing
set $K$. The conjecture of Mumford and Shah predicts that, in the
interior of the domain, $K$ is a locally finite union of $C^1$ arcs,
which may have free endpoints and may meet only in triple junctions
with angles of $120^\circ$. Thus a point of $K$ should have precisely
one of three local configurations: a regular arc, a loose end, or a
triple junction. The difficulty is to derive this description from
minimality without assuming in advance that $K$ is connected, has
finitely many components, or already consists of curves. In this paper
we prove the planar interior conjecture by completing the classification
of generalized global minimizers.

\subsection*{Background and earlier results}

The existence theory rests on formulations that allow discontinuity
sets with no prescribed smoothness. De Giorgi and Ambrosio
\cite{DGA88} introduced the space of special functions of bounded
variation, and Ambrosio \cite{Amb90} developed an existence theory for
free discontinuity energies in this setting. The density estimates
and essential closedness theorem of De Giorgi, Carriero and Leaci
\cite{DGCL89} link the weak formulation to minimization over pairs
$(u,K)$ with $K$ closed. Dal Maso, Morel and Solimini \cite{DMMS92}
developed a complementary geometric approach to existence and
approximation. These results provide a framework in which the
regularity conjecture can be studied without building its conclusion
into the admissible class; see also \cite{AFP}.

A major part of the subsequent theory consists of
$\varepsilon$-regularity results: sufficiently good approximation by
an expected model at one scale implies regularity at smaller scales.
David \cite{Dav96} established such results near straight lines and
triple junctions. The work of Ambrosio--Pallara \cite{AP97} and
Ambrosio--Fusco--Pallara \cite{AFP97} developed partial regularity
in arbitrary dimension. In particular, the discontinuity set of a
planar minimizer is locally a regular arc at $\mathcal H^1$-almost
every point. This almost-everywhere description does not by itself
control the exceptional points or exclude the accumulation of small
connected components.

The loose-end case required a further development. Even when the set
was known to consist of a single arc ending at an interior point,
the existence of a tangent at that endpoint was a separate issue.
De Lellis, Focardi and Ghinassi \cite{DLFG21} proved an
$\varepsilon$-regularity theorem near the crack-tip model, establishing
regularity up to the loose end and vanishing curvature there; see also
the earlier work of Andersson and Mikayelyan \cite{AM19}.
The treatment of the fidelity term, together with a unified account
of the planar theory, is given in the monograph of De Lellis and
Focardi \cite{DLF}. Earlier systematic accounts include \cite{Dav05}.
Another line of progress concerns higher integrability: De Lellis and
Focardi \cite{DLF13} proved that $\nabla u$ belongs locally to $L^p$
for some $p>2$ in the plane, and De Philippis and Figalli
\cite{DPF14} established the corresponding result in arbitrary
dimension.

Further contributions to the regularity theory include higher
regularity of the regular part of the discontinuity set
\cite{AFP99}, flatness and finiteness results \cite{Leg99},
quantitative density lower bounds \cite{LF13}, and monotonicity
formulae for more general free discontinuity problems \cite{BL14}.
For planar minimizers of the fidelity-free functional, Fischer
\cite{Fis22} obtained a monotonicity formula for the David--L\'eger
entropy truncated at the sharp level $3/2$, without a connectedness
assumption on the discontinuity set, and derived an optimal lower bound
on the energy density at points where $u$ is not smooth.
In dimension three, regularity near minimal cones \cite{Lem11},
rigidity of global minimizers \cite{Lem15}, and stationarity of
crack-front configurations \cite{LM18} have also been investigated.
More recently, Zhang \cite{Zha25} gave an alternative proof of a local
square-root growth estimate at tip points, extending an estimate of
David and L\'eger. Labourie and Lemenant \cite{LL26} proved that a planar
minimizer $u$ has at most three distinct limit values at each point
of $K$, when approached from $\Omega\setminus K$, answering a
question of De Giorgi.

The connection between local regularity and global geometry comes
from blow-up analysis. The rescaling
\[
 K_{x,r}=\frac{K-x}{r},\qquad
 u_{x,r}(y)=r^{-1/2}u(x+ry)
\]
puts the Dirichlet energy and the length on the same scale:
\[
 \mathcal E_0(u_{x,r},K_{x,r};B_R)
 =r^{-1}\mathcal E_0(u,K;B_{rR}(x)).
\]
The fidelity term disappears in the limit as $r\downarrow0$.
Compactness, with suitable normalizations of the additive constants
on different complementary components, leads to generalized global
minimizers of the fidelity-free energy. The expected models are the
constant configuration, the pure jump across a straight line, the
triple junction, and the crack tip. The latter consists of a half-line
and a function with square-root growth; its precise normalization is
recalled in \eqref{intro:canonical}. The compactness and
$\varepsilon$-regularity theory makes the interior conjecture
equivalent to the assertion that these models exhaust the generalized
global minimizers \cite[Theorem~1.4.5]{DLF}.

Several fundamental classification results establish this assertion
under additional geometric hypotheses. Bonnet \cite{Bon96} introduced
a monotonicity argument and classified global minimizers whose
discontinuity set is connected. Bonnet and David \cite{BD01} proved
that the crack tip is globally minimizing and obtained a stronger
rigidity theorem, which applies when all but at most one connected
component of $K$ are contained in a common compact set. David and
L\'eger \cite{DL02} classified the case in which $K$ disconnects the
plane, leaving only the straight line and the triple junction, and
showed that $K$ has at most one unbounded connected component.
The forms of these results used below are recalled in
Section~\ref{sec:reduction}. Together they isolate the remaining
difficulty in configurations with disconnected $K$ and connected
complement, for which bounded components cannot all be confined to
one compact set.

There is also a related regularity problem at the boundary of the
ambient domain. For prescribed Dirichlet data, the author
\cite{Dea24} studied boundary blow-ups and proved an
$\varepsilon$-regularity theorem near the tangent line to the
boundary. The result proved here concerns interior points; the
boundary problem involves a different class of limiting minimizers.

\subsection*{Main result}

We now specify the local minimality convention in the statement.
We call $(u,K)$ \emph{admissible} if its energy is locally finite, that is,
$\mathcal E_\lambda(u,K;U)<\infty$ for every open $U\Subset\Omega$.
We call an admissible pair $(u,K)$ an \emph{absolute minimizer} if,
for every bounded open $U\subset\Omega$,
\[
 \mathcal E_\lambda(u,K;U)\leq\mathcal E_\lambda(v,L;U)
\]
for every admissible $(v,L)$ for which some compact $C\subset U$ satisfies
\begin{equation}\label{intro:comparison}
 L\setminus C=K\setminus C,\qquad
 v=u\quad\text{a.e. on }\Omega\setminus(C\cup K\cup L).
\end{equation}
This is absolute minimality in
\cite[Section~1.1 and Definition~1.2.1]{DLF}; competitor sets $L$ are
rectifiable, but there is no restriction on their topology.
We call $K$ \emph{reduced} if
\begin{equation}\label{intro:reduced}
 \mathcal H^1(K\cap B_r(x))>0
 \quad\text{whenever }x\in K\text{ and }\overline{B_r(x)}\subset\Omega.
\end{equation}
When $\Omega=\R^2$, the same definition is used in every ball.
This condition requires $K$ to coincide locally with the support of
its length measure, excluding extraneous points that would obscure
a geometric statement.

\begin{theorem}[Planar interior regularity]\label{main:regularity}
Let $(u,K)$ be an absolute minimizer of \eqref{intro:energy}, and suppose that $K$ is reduced. For every $x\in K$ there are $r>0$ with $\overline{B_r(x)}\subset\Omega$, an integer $m\in\{1,2,3\}$, positive numbers $\ell_1,\ldots,\ell_m$, and injective maps
\[
 \gamma_i\in C^1([0,\ell_i];\R^2),\qquad i=1,\ldots,m,
\]
with
\begin{gather}
 |\gamma_i'(s)|=1\quad(0\leq s\leq\ell_i),\qquad
 \gamma_i(0)=x,\label{intro:arcs}\\
 \gamma_i(\ell_i)\in\partial B_r(x),\qquad
 \gamma_i([0,\ell_i))\subset B_r(x),\notag\\
 K\cap\overline{B_r(x)}=\bigcup_{i=1}^m\gamma_i([0,\ell_i]),\label{intro:star}\\
 \gamma_i([0,\ell_i])\cap\gamma_j([0,\ell_j])=\{x\}\qquad(i\neq j).\notag
\end{gather}
Endpoint derivatives are one-sided. The outward unit tangents $\tau_i=\gamma_i'(0)$ satisfy
\begin{equation}\label{intro:angles}
 \tau_1+\tau_2=0\quad(m=2),\qquad
 \tau_i\cdot\tau_j=-\frac12\quad(i\neq j,\ m=3).
\end{equation}
The set of points with $m=1$ or $m=3$ is locally finite in $\Omega$.
\end{theorem}

\subsection*{Outline of the proof}

The proof reduces to the classification of reduced generalized global
minimizers. By the established results recalled above, it remains to
consider noncompact $K$ with connected complement $\R^2\setminus K$.
Our strategy is to use second variations to exclude bounded connected
components of $K$. Indeed, the David--L\'eger theorem gives at most one
unbounded component; once every component is unbounded, $K$ is connected
and the known classification applies.

A bounded connected component need not be isolated from the other
components, so it cannot necessarily be translated independently.
However, Lemma~\ref{core:topology} shows that it is contained in a
nonempty compact subset $S\subset K$ which is both open and closed
relative to $K$. Since $K$ is noncompact, $S$ is a proper subset of $K$.
Moreover, compactness of $S$ and closedness of $K\setminus S$ imply that
the two sets have positive distance. We can therefore translate $S$ while
keeping the rest of $K$ fixed, preserving its length.
Proposition~\ref{core:rigidity} rules out such a set by combining the
second-variation inequalities for two perpendicular translations,
while allowing $u$ to adjust to each deformation.

Our starting point, Lemma~\ref{core:optimization}, is a
second-variation estimate for a generalized global minimizer with
connected complement $D=\R^2\setminus K$, valid for every deformation
field $X\in C_c^\infty(\R^2;\R^2)$. No smoothness of $K$ in the region
affected by the deformation is required. We consider the combined inner
and outer variations
\[
 \Phi_t=\operatorname{Id}+tX,\qquad K_t=\Phi_t(K),\qquad
 u_t\circ\Phi_t=u+t\varphi,
\]
where $\varphi\in\mathscr T_K$ and
\[
 \mathscr T_K=\{\varphi\in W^{1,2}(D):
       \varphi=0\text{ a.e. outside some compact subset of }\R^2\}.
\]
The coefficient of $t^2$ in the energy expansion is
\[
 Q_X+\|\nabla\varphi\|_2^2
       +2\langle m_X,\nabla\varphi\rangle.
\]
Here $Q_X$ is the coefficient for the pure inner variation, obtained by
setting $\varphi=0$, while $m_X$ couples the deformation to the correction
of $u$ through the inner product $\langle m_X,\nabla\varphi\rangle$ above.
Explicitly,
\begin{align*}
 m_X&=\bigl((\operatorname{div}X)I-DX-(DX)^T\bigr)\nabla u,\\
 Q_X&=\int_D\bigl(|(DX)^T\nabla u|^2-\det(DX)|\nabla u|^2\bigr)\,\dd x\\
 &\quad+\frac12\int_K\bigl(|DX\,\tau|^2-(\tau\cdot DX\,\tau)^2\bigr)
                     \,\dd\mathcal H^1,
\end{align*}
where $\tau$ is a unit approximate tangent to $K$. Norms and inner
products are taken in $L^2(\R^2;\R^2)$.

For each fixed $X$, inner variations alone give $Q_X\geq0$.
The freedom to choose $\varphi$ gives a stronger condition: completing
the square and taking the infimum over $\varphi\in\mathscr T_K$ yields
\[
 Q_X\geq\|\Pi_Km_X\|_2^2,
\]
where $\Pi_K$ is the orthogonal projection onto the $L^2$ closure of
the gradients of functions in $\mathscr T_K$. Thus the optimization
quantifies how much the quadratic cost of a prescribed deformation
can be reduced by adjusting $u$.

Second variations of the Mumford--Shah functional were also studied by
Cagnetti, Mora and Morini \cite{CMM08}. Their analysis concerns
deformations supported in the regular part of the discontinuity set,
with the function chosen to minimize the Dirichlet energy on the
deformed domain with fixed boundary values.

We then apply this general estimate to $X_1=\chi e_1$ and
$X_2=\chi e_2$, where $\chi$ is a smooth compactly supported cutoff
equal to one near $S$ and zero near $K\setminus S$. These fields
generate the two perpendicular translations of $S$. Their pure inner
quadratic costs $Q_{X_1}$ and $Q_{X_2}$ are already nonnegative, so the essential information
comes from the projection term obtained by optimizing over $\varphi$.
The associated fields $m_{X_1}$ and $m_{X_2}$ are related by a
$90^\circ$ rotation and satisfy
\[
 Q_{X_1}+Q_{X_2}=\|m_{X_1}\|_2^2=\|m_{X_2}\|_2^2.
\]
Let $\Pi$ denote the orthogonal projection onto the $L^2$ closure of
the gradients of functions in $C_c^\infty(\R^2)$. The Hodge
decomposition on $\R^2$ gives
\[
 \|\Pi m_{X_1}\|_2^2+\|\Pi m_{X_2}\|_2^2
 =Q_{X_1}+Q_{X_2}.
\]
Since the restrictions to $D$ of functions in $C_c^\infty(\R^2)$
belong to $\mathscr T_K$, the space defining $\Pi$ is contained in
the space defining $\Pi_K$. Thus replacing $\Pi$ by $\Pi_K$ can
only increase each squared norm. On the other hand, the
second-variation inequalities bound $\|\Pi_Km_{X_i}\|_2^2$ by
$Q_{X_i}$ for $i=1,2$. Summing these inequalities forces
equality throughout. Let $v$ be a real scalar potential such that
\[
 \nabla v=\Pi m_{X_1}.
\]
It satisfies $\Delta v=\operatorname{div}m_{X_1}$, so it is harmonic
near $K$, where $m_{X_1}=0$. The equality conditions imply that both
the normal and tangential derivatives of $v$ vanish along every
regular arc of $K$. Thus $\nabla v=0$ on these arcs.

To connect this information back to $u$, we use the equation for the
potential to split one derivative of $u$ into two harmonic functions.
One extends through $S$ and the other through $K\setminus S$.
The vanishing conditions and unique continuation make both functions
constant. It follows that $\nabla u$ is constant, and the linear
energy-growth bound forces it to vanish. The classification of
minimizers with vanishing Dirichlet energy then contradicts the
presence of $S$.

This excludes bounded connected components in the remaining case
and completes the classification in Theorem~\ref{core:classification}.
The planar regularity statement, Theorem~\ref{main:regularity}, then
follows by the existing compactness and regularity theory, including
regularity at loose ends.

Section~\ref{sec:reduction} records the precise results from that
theory needed for this implication. Section~\ref{sec:core} proves
compact translation rigidity and its topological consequence, and
completes the global classification.

\section{Reduction to a global classification}\label{sec:reduction}

All references to the established theory below use the published
monograph \cite{DLF}.

We use \emph{generalized global minimizer} exactly as in
\cite[Definition~2.2.4]{DLF}: it is a triple $(u,K,\{p_{kl}\})$ arising
from rescaled absolute minimizers. The extended-real parameters $p_{kl}$
record relative additive constants on the connected components of
$\R^2\setminus K$. They are part of the definition; when
$\R^2\setminus K$ is connected, the only parameter is $p_{11}=0$.

The compactness theorem below requires a uniform $L^\infty$ bound on the
minimizers before rescaling \cite[Assumption~2.2.1(iv)]{DLF}.
For blow-ups of a fixed absolute minimizer $(u,K)$ at an interior
point, we obtain this bound by working on a sufficiently small ball.
The maximum principle \cite[Lemma~2.1.1]{DLF} gives local boundedness
of $u$ when $\lambda>0$. When $\lambda=0$, a local truncation gives
a bounded function $\widetilde u$ such that $(\widetilde u,K)$ is still
an absolute minimizer and $\nabla\widetilde u=\nabla u$.
We apply the compactness theorem to this bounded minimizer.

The compactness theorem \cite[Theorem~2.2.3]{DLF} supplies a generalized
global limit $(u,K,\{p_{kl}\})$ along a subsequence of every sequence of
interior rescalings. It gives local Hausdorff convergence to $K$ and
separate convergence of Dirichlet energies and length measures. The
limiting set $K$ is reduced, countably $\mathcal H^1$-rectifiable, and
locally of finite length, while $u\in W^{1,2}(B_R\setminus K)$ for
every $R>0$.
For every generalized global minimizer $(u,K,\{p_{kl}\})$, the upper density bound
\cite[Lemma~2.1.2]{DLF} gives
\begin{equation}\label{core:growth}
 \int_{B_R(z)\setminus K}|\nabla u|^2
       +\mathcal H^1(K\cap B_R(z))\leq 2\pi R
 \qquad(z\in\R^2,\ R>0).
\end{equation}
Moreover, \cite[Corollary~3.1.5]{DLF} gives a relatively open subset
of regular points of $K$ of full $\mathcal H^1$ measure. At each such
point $x$, there is $r>0$ such that $K\cap B_r(x)$ is a single
embedded $C^1$ arc.

If $D=\R^2\setminus K$ is connected, the exterior separation condition
on competitors in \cite[Definition~2.2.2]{DLF} is vacuous.
Consequently \cite[Theorem~2.2.3(ii)]{DLF} makes $(u,K)$ an ordinary
absolute minimizer of the fidelity-free energy on the plane.
In particular, the compactly supported combined inner and outer
variations used in Section~\ref{sec:core} are admissible.

A generalized global minimizer $(u,K,\{p_{kl}\})$ is called
\emph{elementary} if $\nabla u=0$ almost everywhere on
$\R^2\setminus K$. By \cite[Theorem~2.4.1]{DLF}, in this case $K$
is either empty, a straight line (the pure-jump model), or the union of
three half-lines with a common endpoint and pairwise angles of
$120^\circ$ (the triple-junction model). In each case, $u$ is constant
on every connected component of $\R^2\setminus K$.
The other canonical model is the crack tip. Up to translation,
rotation, sign, and an additive constant, its set $K$ and function
$u$ are given by
\begin{equation}\label{intro:canonical}
\begin{gathered}
 K=\{(r,0):r\geq0\},\\
 u(r,\theta)=\sqrt{\frac2\pi}\sqrt r\cos\frac\theta2,
 \quad r>0,\quad0<\theta<2\pi.
\end{gathered}
\end{equation}
The parameters $p_{kl}$ in the two disconnected models are interpreted
as in \cite[Section~2.4]{DLF}.
For a generalized global minimizer $(u,K,\{p_{kl}\})$, we shall use the
following classification results:
\begin{enumerate}[label=\textup{(\roman*)}]
\item If $\R^2\setminus K$ is disconnected, then $\nabla u=0$ almost
everywhere on $\R^2\setminus K$, and $K$ is either a straight line or
the triple-junction configuration described above
\cite[Theorem~4.4.1]{DLF}.
\item $K$ has at most one unbounded connected component
\cite[Corollary~4.4.3]{DLF}.
\item If all but at most one connected component of $K$ lie in one
common compact set, then the minimizer is elementary or $(u,K)$ is
the canonical crack tip \eqref{intro:canonical}, up to the transformations
specified above \cite[Theorem~1.4.6, first paragraph]{DLF}.
\end{enumerate}
These are the established David--L\'eger and Bonnet--David classification
results, in the form needed for generalized limits.

By \cite[Theorem~1.4.5]{DLF}, it suffices to prove the
classification stated in Theorem~\ref{core:classification}.
After the usual local normalization and rescaling, the
regularity theorem \cite[Theorem~1.3.3]{DLF} gives
Theorem~\ref{main:regularity}.
Along each sufficiently short $C^1$ branch issuing from $x$, the
distance from $x$ increases strictly. We may therefore choose a
smaller ball whose boundary meets each branch exactly once.
Parametrizing the resulting closed arcs by arclength gives
\eqref{intro:arcs}--\eqref{intro:star}.
The local description at every point of $K$ also gives a neighborhood
containing no endpoints or triple junctions other than, possibly,
that point itself. Since $K$ is relatively closed in $\Omega$, such
points cannot accumulate in $\Omega$ and are therefore locally finite.

\section{Compact translations and global classification}\label{sec:core}

We first prove the compact translation rigidity result and then use it to complete the classification of generalized global minimizers.

\subsection{Admissible variations and optimization over corrections}

For a closed set $K\subset\R^2$ of zero two-dimensional Lebesgue measure, put $D=\R^2\setminus K$ and define the real test space
\[
 \mathscr T_K=\{\varphi\in W^{1,2}(D):
       \varphi=0\text{ a.e. outside some compact subset of }\R^2\}.
\]
We shall use the closed subspace
\[
 \mathcal V_K=\overline{\{\nabla\varphi:\varphi\in\mathscr T_K\}}
                   ^{\,L^2(\R^2;\R^2)}
\]
and its orthogonal projection $\Pi_K$. Sets of zero two-dimensional Lebesgue measure are ignored when identifying vector fields on $D$ with vector fields on the plane.

Let $(u,K,\{p_{kl}\})$ be a generalized global minimizer with
$\R^2\setminus K$ connected. Absolute minimality gives the weak equation
\begin{equation}\label{core:weak}
 \int_D\nabla u\cdot\nabla\varphi=0\qquad(\varphi\in\mathscr T_K)
\end{equation}
and the inner variation identity
\begin{equation}\label{core:inner}
 \int_D\bigl(|\nabla u|^2\operatorname{div}X
              -2\nabla u\cdot DX\nabla u\bigr)
 +\int_K\tau\cdot DX\tau\,\dd\mathcal H^1=0.
\end{equation}
Here $X\in C_c^\infty(\R^2;\R^2)$ and $\tau$ is a unit approximate
tangent to $K$. These are the usual variational identities; see
\cite[Section~2.5]{DLF}.
Minimality also gives second-order stability under the combined inner
and outer variations
\begin{equation}\label{core:combined}
 \Phi_t=\operatorname{Id}+tX,\qquad K_t=\Phi_t(K),\qquad
 u_t\circ\Phi_t=u+t\varphi.
\end{equation}
These reduce to inner variations when $\varphi=0$ and to outer
variations when $X=0$.
Differences of the infinite global energies are evaluated on a bounded
region containing the support of the variation.

\begin{proposition}[Compact translation rigidity]\label{core:rigidity}
Let $(u,K,\{p_{kl}\})$ be a reduced generalized global minimizer with
$\R^2\setminus K$ connected. There is no nonempty compact proper subset of $K$
which is both open and closed relative to $K$.
\end{proposition}

\begin{lemma}[Second variation optimized over corrections to $u$]\label{core:optimization}
Let $(u,K,\{p_{kl}\})$ be a generalized global minimizer with
$\R^2\setminus K$ connected.
For $X\in C_c^\infty(\R^2;\R^2)$, define
\begin{align}
 m_X&=\bigl((\operatorname{div}X)I-DX-(DX)^T\bigr)\nabla u,
       \label{core:load}\\
 Q_X&=\int_D\bigl(|(DX)^T\nabla u|^2-\det(DX)|\nabla u|^2\bigr)\,\dd x
       \notag\\
 &\quad+\frac12\int_K\bigl(|DX\,\tau|^2-(\tau\cdot DX\,\tau)^2\bigr)
                  \,\dd\mathcal H^1.\label{core:quadratic}
\end{align}
Then
\begin{equation}\label{core:projectionbound}
 Q_X\geq\|\Pi_Km_X\|_{L^2}^2.
\end{equation}
\end{lemma}

\begin{proof}
Fix $\varphi\in\mathscr T_K$, and choose a ball $B$ containing the
supports of $X$ and $\varphi$ compactly. For sufficiently small $|t|$,
the map $\Phi_t$ in \eqref{core:combined} is a diffeomorphism equal to
the identity outside a compact subset of $B$; in particular,
$\Phi_t(B)=B$.
For the matrix calculation, write $A=DX$.
The identity $u_t\circ\Phi_t=u+t\varphi$ and the Sobolev chain rule give
\[
 \nabla u_t(\Phi_t(x))
   =(I+tA(x))^{-T}\bigl(\nabla u(x)+t\nabla\varphi(x)\bigr)
 \quad\text{for a.e. }x\in B\setminus K.
\]
Changing variables $y=\Phi_t(x)$ in the Dirichlet energy, we obtain
\begin{align*}
 &\int_{B\setminus K_t}|\nabla u_t(y)|^2\,\dd y\\
 &\qquad=\int_{B\setminus K}
    \bigl|(I+tA)^{-T}(\nabla u+t\nabla\varphi)\bigr|^2
       \det(I+tA)\,\dd x\\
 &\qquad=\int_{B\setminus K}
    (\nabla u+t\nabla\varphi)\cdot
       M_t(\nabla u+t\nabla\varphi)\,\dd x,
\end{align*}
where
\begin{align*}
 M_t&=\det(I+tA)(I+tA)^{-1}(I+tA)^{-T}\\
 &=I+t\bigl((\operatorname{tr}A)I-A-A^T\bigr)
       +t^2\bigl(AA^T-\det A\,I\bigr)+O(t^3).
\end{align*}
The second-order coefficient follows from the two-dimensional
Cayley--Hamilton identity
$A^2-(\operatorname{tr}A)A+\det A\,I=0$ and its transpose.

For the length term, the area formula for countably rectifiable sets
(see \cite[Theorem~2.91]{AFP}) gives
\[
 \mathcal H^1(K_t\cap B)
   =\int_{K\cap B}|(I+tA)\tau|\,\dd\mathcal H^1,
\]
where $\tau$ is a unit approximate tangent to $K$.
The tangential Jacobian expands as
\[
 |(I+tA)\tau|
 =1+t\,\tau\cdot A\tau
   +\frac{t^2}{2}\bigl(|A\tau|^2-(\tau\cdot A\tau)^2\bigr)
   +O(t^3).
\]
Since $DX$ is bounded and $|\tau|=1$, the $O(t^3)$ remainder
estimates in the expansions of $M_t$ and the tangential Jacobian are
uniform in the spatial variable. Since $u$ and $\varphi$ have finite
Dirichlet energy on $B\setminus K$ and $\mathcal H^1(K\cap B)<\infty$,
the integrated remainder is $o(t^2)$ for each fixed $\varphi$.
Since the competitors are admissible for both signs of sufficiently
small $t$, minimality implies that the differentiable function
$t\mapsto\mathcal E_0(u_t,K_t;B)$ has a local minimum at $t=0$.
Thus its linear coefficient vanishes, and the preceding expansions give
\begin{align*}
 &\mathcal E_0(u_t,K_t;B)-\mathcal E_0(u,K;B)\\
 &\qquad=t^2\bigl(Q_X+\|\nabla\varphi\|_2^2
                   +2\langle m_X,\nabla\varphi\rangle\bigr)+o(t^2),
\end{align*}
where norms and inner products are in $L^2(\R^2;\R^2)$.
Dividing the minimality inequality by $t^2$ and letting $t\to0$ yields
\[
 Q_X+\|\nabla\varphi\|_2^2
       +2\langle m_X,\nabla\varphi\rangle\geq0
 \qquad(\varphi\in\mathscr T_K).
\]

Since $\nabla\varphi\in\mathcal V_K$, orthogonal projection gives
$\langle m_X,\nabla\varphi\rangle
 =\langle\Pi_Km_X,\nabla\varphi\rangle$. Completing the square,
\[
 \|\nabla\varphi\|_2^2+2\langle m_X,\nabla\varphi\rangle
   =\|\nabla\varphi+\Pi_Km_X\|_2^2-\|\Pi_Km_X\|_2^2.
\]
By the definition of $\mathcal V_K$, there are
$\varphi_j\in\mathscr T_K$ such that
$\nabla\varphi_j\to-\Pi_Km_X$ in $L^2(\R^2;\R^2)$.
Applying the preceding inequality to $\varphi_j$ and passing to the
limit proves \eqref{core:projectionbound}.
The limit $t\to0$ is taken first for each fixed test, so no remainder
estimate uniform in $j$ is needed. In particular, the projected field
need not itself be the gradient of an element of $\mathscr T_K$.
\end{proof}

\subsection{Two translations and the Hodge decomposition}

\begin{proof}[Proof of Proposition~\ref{core:rigidity}]
Suppose that $S\subsetneq K$ is nonempty, compact and relatively open and closed. Its nonempty complement $K\setminus S$ is closed and has positive distance from $S$. We may therefore choose a real $\chi\in C_c^\infty(\R^2)$ such that
\[
 \chi=1\quad\text{near }S,\qquad
 \chi=0\quad\text{near }K\setminus S.
\]
Thus $\operatorname{supp}\nabla\chi$ is a compact subset of $D$. Reducedness and relative openness imply
\begin{equation}\label{core:positive}
 0<\mathcal H^1(S)<\infty.
\end{equation}

Set $X_1=\chi e_1$ and $X_2=\chi e_2$. The corresponding deformations translate $S$ rigidly and are the identity in a neighborhood of $K\setminus S$. Let $\Rot(a,b)=(-b,a)$. Substituting into \eqref{core:load}--\eqref{core:quadratic} gives
\begin{align}
 m:=m_{X_1}&=(-\chi_xu_x-\chi_yu_y,
                         -\chi_yu_x+\chi_xu_y),\notag\\
 m_{X_2}&=\Rot m,\label{core:twoloads}\\
 Q_{X_1}&=\int_Du_x^2|\nabla\chi|^2,\qquad
 Q_{X_2}=\int_Du_y^2|\nabla\chi|^2.\notag
\end{align}
Consequently
\begin{equation}\label{core:totalcost}
 Q_{X_1}+Q_{X_2}=\|m\|_2^2.
\end{equation}
The field $m$ is smooth and compactly supported in $D$, since \eqref{core:weak} makes $u$ harmonic there.

Let $\mathcal V$ be the $L^2$ closure of gradients of real functions in
$C_c^\infty(\R^2)$, and let $\Pi$ be its orthogonal projection. Then
$\mathcal V\subset\mathcal V_K$. The planar $L^2$ Helmholtz--Hodge
decomposition (see \cite[Section~III.1]{Gal11}) gives
\[
 L^2(\R^2;\R^2)=\mathcal V\oplus^\perp\Rot\mathcal V.
\]
Write
\[
 m=g+\Rot h,\qquad g,h\in\mathcal V.
\]
Since $\mathcal V$ and $\Rot\mathcal V$ are orthogonal in $L^2$ and
$\Rot^2=-I$, we have
\[
 \Pi m=g,\qquad \Pi(\Rot m)=\Pi(\Rot g-h)=-h.
\]
Therefore, using that $\Rot$ preserves the $L^2$ norm,
\begin{equation}\label{core:hodge}
 \begin{aligned}
 \|\Pi m\|_2^2+\|\Pi(\Rot m)\|_2^2
 &=\|g\|_2^2+\|h\|_2^2\\
 &=\|g\|_2^2+\|\Rot h\|_2^2\\
 &=\|m\|_2^2.
 \end{aligned}
\end{equation}
Since $\mathcal V\subset\mathcal V_K$, for every
$f\in L^2(\R^2;\R^2)$ we have $\Pi_Kf-\Pi f\perp\mathcal V$, and hence
\[
 \|\Pi_Kf\|_2^2-\|\Pi f\|_2^2
 =\|\Pi_Kf-\Pi f\|_2^2\geq0.
\]
Combining \eqref{core:projectionbound}, \eqref{core:totalcost} and \eqref{core:hodge}, we obtain
\begin{align*}
 \|m\|_2^2
 &\geq\|\Pi_Km\|_2^2+\|\Pi_K(\Rot m)\|_2^2\\
 &\geq\|\Pi m\|_2^2+\|\Pi(\Rot m)\|_2^2
 =\|m\|_2^2.
\end{align*}
The first and last quantities coincide, so equality holds throughout.
Using the preceding projection identity, we obtain
\[
 \|\Pi_Km-\Pi m\|_2^2
 +\|\Pi_K(\Rot m)-\Pi(\Rot m)\|_2^2=0.
\]
Both terms are nonnegative, so
\[
 \Pi_Km=\Pi m,\qquad \Pi_K(\Rot m)=\Pi(\Rot m).
\]
Since $f-\Pi_Kf\perp\mathcal V_K$ for every $f$, we conclude that
\begin{equation}\label{core:residuals}
 R_1=m-\Pi m\perp\mathcal V_K,
 \qquad R_2=\Rot m-\Pi(\Rot m)\perp\mathcal V_K.
\end{equation}

\subsubsection*{A real scalar potential and the normal conditions}
Since $\Pi m\in\mathcal V$, it is an $L^2$ limit of gradients of smooth,
compactly supported functions, and hence $\operatorname{curl}(\Pi m)=0$
in $\mathcal D'(\R^2)$. As the plane is simply connected, the Poincar\'e
lemma supplies a real function $v\in W^{1,2}_{\mathrm{loc}}(\R^2)$ such that
\begin{equation}\label{core:potential}
 \nabla v=\Pi m,\qquad \Delta v=\operatorname{div}m
 \quad\text{in }\mathcal D'(\R^2).
\end{equation}
The second identity follows because $m-\Pi m$ is orthogonal to all
compactly supported smooth gradients. Since $m\in C_c^\infty(D;\R^2)$,
elliptic regularity gives $v\in C^\infty(\R^2)$, and $v$ is harmonic
on a neighborhood of $K$.

Using the decomposition above, \eqref{core:residuals} and
\eqref{core:potential} give
\[
 R_1=m-\nabla v=\Rot h,\qquad R_2=\Rot g=\Rot\nabla v.
\]
In particular, both residuals are smooth and divergence-free on the
plane, and, near $K$, where $m=0$,
\begin{equation}\label{core:rotation}
 R_1=-\nabla v,\qquad R_2=\Rot\nabla v.
\end{equation}

Fix a regular point of $K$ and choose, in suitable coordinates, a
sufficiently small rectangle $U$ around it such that
$\Gamma:=K\cap U$ is a single $C^1$ graph. The set $U\setminus K$
is the union of two Lipschitz domains $U^+$ and $U^-$ lying on either
side of $\Gamma$. Let $\nu$ denote the outward unit normal to $U^+$.
For any real $\eta\in C_c^\infty(U)$, define on $D$
\[
 \varphi(x)=
 \begin{cases}
  \eta(x),&x\in U^+,\\
  0,&x\in D\setminus U^+.
 \end{cases}
\]
Since $\eta$ vanishes near $\partial U$, this definition creates no
jump inside $D$: the only possible jump lies on $\Gamma\subset K$.
Thus $\varphi\in W^{1,2}(D)$, with
$\nabla\varphi=\mathbf 1_{U^+}\nabla\eta$ almost everywhere on $D$.
Its support in the plane is contained in $\supp\eta\Subset U$, so
$\varphi\in\mathscr T_K$.

Using \eqref{core:residuals} and integration by parts on $U^+$,
we obtain, for $j=1,2$,
\[
 \begin{aligned}
  0&=\int_D R_j\cdot\nabla\varphi\,\dd x
    =\int_{U^+}R_j\cdot\nabla\eta\,\dd x\\
   &=\int_{\partial U^+}\eta\,R_j\cdot\nu\,\dd\mathcal H^1
       -\int_{U^+}\eta\,\operatorname{div}R_j\,\dd x\\
   &=\int_\Gamma\eta\,R_j\cdot\nu\,\dd\mathcal H^1.
 \end{aligned}
\]
Here the boundary term on $\partial U^+\setminus\Gamma$ vanishes
because $\eta$ is supported compactly in $U$, and the volume term
vanishes because $\operatorname{div}R_j=0$.
Since the identity holds for every $\eta\in C_c^\infty(U)$ and
$R_j\cdot\nu$ is continuous on $\Gamma$, we conclude that
\[
 R_1\cdot\nu=R_2\cdot\nu=0\quad\text{on }\Gamma.
\]
By \eqref{core:rotation},
\[
 \nabla v\cdot\nu=(\Rot\nabla v)\cdot\nu=0.
\]
The vectors $\nu$ and $\Rot\nu$ span the plane, so
\begin{equation}\label{core:arczeros}
 \nabla v=0\quad\text{on every regular subarc of }K.
\end{equation}

\subsubsection*{Harmonic separation of the compact piece}
We use the following standard harmonic uniqueness fact: if $w$ is harmonic
on a connected open set $U\subset\R^2$ and $\nabla w=0$ along a
nontrivial embedded $C^1$ arc contained in $U$, then $w$ is constant.
Indeed, $w=c$ and $\partial_\nu w=0$ on a connected subarc.
Extending $w-c$ from one side by zero to the other gives a weakly
harmonic function in a small neighborhood, by integration by parts.
Interior elliptic regularity and the identity principle imply that
this extension vanishes. Thus $w=c$ on an open subset of $U$, and hence
throughout $U$.

Using $u_{yy}=-u_{xx}$ and $u_{xy}=u_{yx}$ in
\eqref{core:twoloads}, we compute on $D$
\begin{equation}\label{core:loaddivergence}
 \operatorname{div}m
 =-(\Delta\chi)u_x-2\nabla\chi\cdot\nabla u_x
 =-\Delta(\chi u_x).
\end{equation}
Here $u_x$ is harmonic on $D$. By \eqref{core:potential} and
\eqref{core:loaddivergence}, the real scalar functions
\begin{equation}\label{core:split}
 f_0=(1-\chi)u_x-v,\qquad f_1=\chi u_x+v
\end{equation}
are harmonic on $D$ and satisfy $f_0+f_1=u_x$.
Near $S$, the first equals $-v$, so it extends harmonically through $S$.
Near $K\setminus S$, the second equals $v$, so it extends harmonically
through $K\setminus S$. Their respective domains of extension are
\[
 U_0=\R^2\setminus(K\setminus S),\qquad
 U_1=\R^2\setminus S.
\]
Both domains are connected: each contains connected $D$, and every
open component meets $D$ because $K$ has zero two-dimensional Lebesgue measure.

By \eqref{core:positive} and full-measure regularity, $S$ contains a
regular subarc. Near that subarc $f_0=-v$, and
\eqref{core:arczeros} gives $\nabla f_0=0$ there. The harmonic
uniqueness fact therefore makes $f_0$ constant throughout $U_0$.
The nonempty relatively open set $K\setminus S$ also contains a
positive-length regular subarc, by reducedness and full-measure
regularity. Near it $f_1=v$, so the same argument makes $f_1$
constant throughout $U_1$.

It follows that $u_x=f_0+f_1$ is constant on $D$. Harmonicity gives
\[
 \nabla u_y=(u_{xy},u_{yy})=(0,-u_{xx})=0.
\]
Connectedness of $D$ then makes $u_y$ constant as well. Thus
$\nabla u$ is a constant vector on $D$. Since $K$ has zero two-dimensional Lebesgue
measure, a nonzero constant gradient would give quadratic
bulk-energy growth, contrary to \eqref{core:growth}. Consequently
$\nabla u=0$ on $D$. By the classification of elementary generalized
global minimizers \cite[Theorem~2.4.1]{DLF}, connectedness of $D$
implies $K=\varnothing$, contradicting $\varnothing\neq S\subset K$.
This proves the proposition.
\end{proof}

\subsection{Topological consequence and classification}

\begin{lemma}\label{core:topology}
Let $K\subset\R^2$ be closed. If $K$ has no nonempty compact subset
that is both relatively open and relatively closed in $K$, then every
connected component of $K$ is unbounded.
\end{lemma}

\begin{proof}
Suppose that $C$ is a bounded connected component of $K$. Since $C$
is closed in $K$ and $K$ is closed in $\R^2$, the set $C$ is closed
and bounded in $\R^2$, hence compact. Choose $R$ with $C\subset B_R$
and set
\[
 X:=K\cap\overline B_R,\qquad E:=X\cap\partial B_R.
\]
The space $X$ is compact, and $C\cap E=\varnothing$. Moreover, $C$
is a connected component of $X$: it is connected and contained in $X$,
and any connected subset of $X$ strictly containing $C$ would also be
a connected subset of $K$, contradicting the maximality of $C$.

We use the standard fact \cite[Theorem~6.1.23]{Eng89}
that a component of a compact Hausdorff
space is the intersection of all subsets containing it that are both
open and closed in that space. Thus, for each $y\in E$, there is a set
$V_y\subset X$, open and closed in $X$, such that $C\subset V_y$ and
$y\notin V_y$. The sets $X\setminus V_y$ form an open cover of $E$
relative to $X$. Since $E$ is compact, we can choose finitely many
$y_1,\ldots,y_N\in E$ such that
\[
 E\subset\bigcup_{i=1}^N(X\setminus V_{y_i}),
 \qquad W:=\bigcap_{i=1}^N V_{y_i}.
\]
Then $W$ is open and closed in $X$ and contains $C$. The covering
property gives $W\cap E=\varnothing$, so $W\subset B_R$. If
$E=\varnothing$, take $W=X$, which has the same properties.

Since $W$ is closed in the compact space $X$, it is compact and hence
closed in $\R^2$, so it is also relatively closed in $K$. Since $W$ is
open in $X$, write $W=O\cap X$ for some open set $O\subset\R^2$.
As $W\subset B_R$, we have
\[
 W=K\cap(O\cap B_R),
\]
so $W$ is relatively open in $K$. Finally, $W$ is nonempty because it contains $C$.
Thus $W$ is a nonempty compact subset of $K$ that is both relatively
open and relatively closed in $K$, contrary to the hypothesis.
\end{proof}

\begin{theorem}[Classification of generalized global minimizers]\label{core:classification}
\leavevmode\par\nobreak\noindent
Let $(u,K,\{p_{kl}\})$ be a reduced generalized global minimizer in
$\R^2$. Then either $\nabla u=0$ a.e.\ on $\R^2\setminus K$, so that the
minimizer is elementary, or $(u,K)$ is the canonical crack tip
\eqref{intro:canonical}, up to translation, rotation, sign, and an additive
constant.
\end{theorem}

\begin{proof}
Let $(u,K,\{p_{kl}\})$ be a generalized global minimizer in the sense of
\cite[Definition~2.2.4]{DLF}.
If $\R^2\setminus K$ is disconnected,
\cite[Theorem~4.4.1]{DLF} gives a pure-jump or triple-junction model.
If $K$ is compact, including $K=\varnothing$,
\cite[Theorem~1.4.6, first paragraph]{DLF} already applies, since all
components are contained in a common compact set. Its only model with
compact $K$ is the constant minimizer.

It remains to consider noncompact $K$ with connected complement.
Proposition~\ref{core:rigidity} excludes every nonempty compact relatively
open-and-closed subset of $K$, since any such subset would be proper.
Lemma~\ref{core:topology} therefore makes every component of $K$
unbounded. By \cite[Corollary~4.4.3]{DLF} there is at most one such
component, so $K$ is connected. The classification of generalized global
minimizers with connected $K$ \cite[Theorem~1.4.6]{DLF} gives the
conclusion.
\end{proof}

\section*{Acknowledgements}
Funded by the Deutsche Forschungsgemeinschaft (DFG, German Research
Foundation) under Germany's Excellence Strategy EXC~2044/2 -- 390685587,
Mathematics M\"unster: Dynamics--Geometry--Structure.

The initial proof strategy and manuscript draft were generated using
ChatGPT Astra (OpenAI), through an iterative exchange in which the
author supplied background material and earlier work. ChatGPT was also
used during the subsequent revision and simplification of the
mathematical arguments and their exposition. The author has revised
the manuscript and takes full responsibility for its mathematical
content and references. An account of the development of the work and
its chronology appears in the Author's note in version~1 of this
preprint, \href{https://arxiv.org/abs/2609.26732v1}{arXiv:2609.26732v1}.

\end{document}